\documentclass[reqno]{amsart}
\usepackage{amssymb,amsthm,amsmath,graphicx,pinlabel}
\usepackage{fullpage}
\pdfoutput=1
\usepackage{xr-hyper}
\usepackage[colorlinks]{hyperref}
\usepackage[style=default, margin=0pt, parskip=0pt, hangindent=0pt, indention=0pt, singlelinecheck=true, labelfont=up]{subcaption}
\usepackage{todonotes}

\DeclareMathOperator{\crit}{crit}

\DeclareMathOperator{\Rp}{Re}
\DeclareMathOperator{\Cp}{Im}

\DeclareMathOperator{\mcg}{MCG}
\DeclareMathOperator{\Kr}{Ker}

\newcommand{\co}{\colon\thinspace}
\newcommand{\C}{\mathbb{C}}
\newcommand{\R}{\mathbb{R}}
\newcommand{\Z}{\mathbb{Z}}

\newcommand{\MYhref}[3][blue]{\href{#2}{\color{#1}{#3}}}

\theoremstyle{definition}
\newtheorem{thm}{Theorem}[section]
\newtheorem{lemma}[thm]{Lemma}

\newtheorem{ex}[thm]{Example}
\newtheorem{rmk}[thm]{Remark}
\newtheorem{prop}[thm]{Proposition}

\newtheorem{df}[thm]{Definition}

\title{Existence of two-parameter crossings, with applications}
\author[Jonathan D. Williams]{Jonathan D. Williams} 
\address{Department of Mathematical Sciences, Binghamton University}
\email{\MYhref{mailto:jdw.math@gmail.com}{jdw.math@gmail.com}}

\begin{document}
\begin{abstract}A Morse 2-function is a generic smooth map $f$ from a manifold $M$ of arbitrary finite dimension to a surface $B$. Its critical set maps to an immersed collection of cusped arcs in $B$. The aim of this paper is to explain exactly when it is possible to move these arcs around in $B$ by a homotopy of $f$ and to give a library of examples when $M$ is a closed 4-manifold. The last two sections give applications to the theory of crown diagrams of smooth 4-manifolds.\end{abstract}
\maketitle
\tableofcontents
\section{Introduction}\label{introduction}The theory of singular fibrations of $n$-manifolds for $n\geq3$ has many applications in low-dimensional topology. For instance Morse functions, open books and Lefschetz fibrations all lead to depictions of 3- and 4-manifolds using surfaces (which are regular fibers) decorated with circles, and such discrete information has been a rich playground. Though its definitions go at least as far back as catastrophe theory, the Morse 2-function as a fibration structure is a promising invention of \cite{ADK}, with recent development such as \cite{AK,B1,B2,Be,BH,GK3,H,Le,W1,W2}. The image of the critical set of a Morse 2-function is an immersed collection of cusped 1-submanifolds, and the main goal of this paper is to explain when it is possible to move the critical image around by a homotopy of the fibration map (Theorem \ref{isotthm}). Some of the results from Section \ref{thmandapps} are entirely straightforward; however, particular applications of Theorem \ref{isotthm} appearing in Proposition \ref{r2prop} are significant for Morse 2-function experts, clearing up a well-known problem known since at least 2008 \cite{GK2} concerning subtle complications around movements that resemble Reidemeister-II moves between fold arcs. We use this knowledge in the proof of Theorem \ref{T} below. Aside from these considerations, and from being a crucial reference of \cite{W2}, the broader and perhaps more important goal of Section \ref{thmandapps} is to give a unified and reasonably universal reference for those who wish to modify Morse 2-functions in the course of constructing smooth 4-manifolds.

The subsequent sections give applications of the ideas from Section \ref{thmandapps}. Section \ref{simpsection} presents an algorithm for converting an arbitrary indefinite Morse 2-function into one that is \emph{simplified} in the sense of \cite{BS}, with a significantly shorter proof. Section 4 has refinements of the main result of \cite{W2}, which is a uniqueness result for crown diagrams of smooth 4-manifolds; see Definition \ref{sddef} (in \cite{W2} these were called \emph{surface diagrams}). One of these refinements states that if two crown diagrams come from homotopic fibration maps and have the same genus, then one may be transformed into the other using genus-preserving moves. This is in contrast to the corresponding statement for pairs of Heegaard diagrams of 3-manifolds, which in general require stabilizations, and it simplifies any effort to define a diffeomorphism invariant calculated using crown diagrams.
\begin{df}An \emph{indefinite Morse 2-function} is a smooth map $f$ from an  $n$-manifold $M$ to a 2-manifold $B$ such that each critical point has the local model of an index $k$ \emph{indefinite fold} or an index $k$ \emph{indefinite cusp}, given by
\[\begin{array}{lc}(x_1,\ldots,x_n)\mapsto\left(x_1,-\sum\limits_{i=2}^{k+1}x_i^2+\sum\limits_{i=k+2}^nx_i^2\right) &\text{(fold)}\\
(x_1,\ldots,x_n)\mapsto\left(x_1,x_2^3-3x_1x_2-\sum\limits_{i=3}^{k+2}x_i^2+\sum\limits_{i=k+3}^nx_i^2\right)\ \ \ \ \ \ \ \ &\text{(cusp)}\end{array}\]
for $k\in\{1,\ldots,n-2\}$. In this paper, a \emph{deformation} of a Morse 2-function is a homotopy of maps from a smooth, closed oriented 4-manifold $M$ to a smooth, closed oriented surface $B$ given by a sequence of merges, births, flips and 2-parameter crossings as described in \cite{W2}.\end{df}
The local models for $k=0$ and $k=n-1$ (the \emph{definite} fold and cusp) are omitted, though Theorem ~\ref{isotthm} could be easily adapted to handle those critical points as well. The reason for this is that the definite locus can be eliminated from smooth maps when $n=4$ by homotopy (see for example the earliest proof of this fact in \cite{S}), and that there are general existence results for indefinite one-parameter families of maps connecting any pair of homotopic indefinite Morse 2-functions when $n\geq4$ \cite{W2,GK3}.
\begin{figure}
    \centering
    \begin{subfigure}{0.3\textwidth}
        \includegraphics{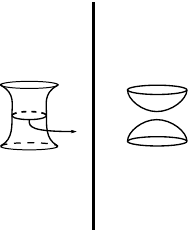}
        \caption{Fold.}
        \label{fold}
    \end{subfigure}
    \qquad\qquad 
    \begin{subfigure}{0.3\textwidth}
        \includegraphics{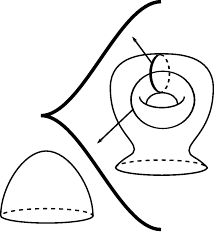}
        \caption{Cusp.}
        \label{cusp}
    \end{subfigure}
    \caption{Critical points of indefinite Morse 2-functions.}\label{crit}
\end{figure}
To help understand folds and cusps, Figure~\ref{crit} shows two disks in $B$, schematically depicting the fibration structure near a fold and a cusp for $n=4$. Bold arcs represent the image of the critical locus and a surface is pictured in the region of regular values that have that surface as their preimage. Tracing point preimages above a horizontal arc from left to right gives the foliation of $\R^3$ by hyperboloids, with a double cone above the fold point. The circle (drawn on the cylinder to the left) that shrinks to the cone point is called the \emph{vanishing cycle} for that critical arc. Each cusp point is a common endpoint of two open arcs of fold points for which the two vanishing cycles must transversely intersect at a unique point in the fiber. See \cite{GK3} for a more detailed background on these maps. 
\begin{figure}[h]
	\centering{\includegraphics{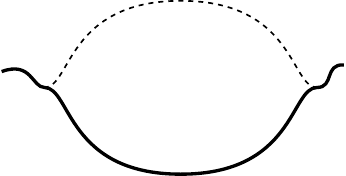}}
	\caption{A fold arc (in bold) runs along the boundary of a disk $\Delta$ in the target space of $f_0$ in preparation to move across $\Delta$ by some deformation of $f_0$, eventually to coincide with the dotted line (other critical points omitted).}\label{isot}
\end{figure}

Section \ref{thmandapps} gives the complete characterization of when a proposed movement in the critical image of an indefinite Morse 2-function, similar to a sequence of Reidemeister moves for knot projections, can be achieved by a deformation. Such homotopies are sequences of \emph{2-parameter crossings}, a term that appeared in \cite[Section~2]{GK3}. In their paper and in this paper, it replaces the less precise term \emph{isotopy} that appeared in \cite{Le} and \cite[Section~2.4.1]{W1}, which in those papers was a sequence of 2-parameter crossings. To describe the general situation, suppose there is a Morse 2-function $f_0\co~M^4\rightarrow~D^2$ and an embedded disk \[\Delta\co\{z\in\C:|z|\leq1\}\rightarrow~D^2\] smoothly embedded except at the two corners on either side, such that, for some fold arc $\varphi_0\subset M_0$, the restriction \[f_0|_{\varphi_0}\co~\varphi_0\rightarrow\Delta\left(\{|z|=1,~\Cp z\leq0\}\right)\] is a diffeomorphism (see Figure~\ref{isot}; sometimes $\Delta$ also denotes its image). It is usually important to keep track of the $t$ parameter when discussing deformations, and in this paper subscripts are reserved for it unless stated otherwise, so for example $\varphi_0$ denotes a critical arc in the initial map $f_0$ of a deformation, and $M_I$ denotes $I\times M$, where $I=[0,1]$. The problem is to determine when there is a deformation $f_t\co M_I\rightarrow D^2_I$ such that the image of each point $\Delta(z)\in~f_0(\varphi_0)$ in this fold arc follows the path $\Delta(t\overline{z}+(1-t)z)$ while the rest of the critical image remains fixed. Theorem~\ref{isotthm} gives a necessary and sufficient criterion for this to occur, and the following propositions present a few standard cases in which a proposed 2-parameter crossing always exists. 

\subsubsection*{Acknowledgments}The author would like to thank Denis Auroux, David Gay and Yank{\i} Lekili for helpful conversations during the preparation of this work.
\section{Existence of 2-parameter crossings}\label{thmandapps}
Consider a Morse function $f\co M^n\to[0,4]$ with two isolated critical points $p_1$ and $p_2$, with $f(p_i)=i$, and equip $M$ with a Riemannian metric. If the ascending manifold $Z$ of $p_1$ is disjoint from the descending manifold $Y$ of $p_2$ (where critical points are considered part of their ascending and descending manifolds), then there is a neighborhood $\nu Z$ on which the restriction of $f$ has exactly one critical point $p_1$, and a homotopy of $f$ supported on $\nu Z$ that sends $f(p_1)$ to $3$. In Figure~\ref{isot}, there is a one-parameter family of smooth functions, finitely many of which may not be Morse, and a one-parameter family $Z$ of ascending manifolds for the one-parameter family of Morse critical points given by $\varphi_0$. If $Z$ is disjoint from the one-parameter family of descending manifolds coming from all other critical points in $\Delta$, then the proposed homotopy exists over every vertical arc $\{\Rp z=const\}$, allowing the fold arc to pass over $\Delta$. Here follows a more precise statement using the language of \cite[Section~2.3]{BH}. To informally explain the notation in Definition \ref{dsjtcond}, given a smooth embedded path $\gamma_p$ in $\Delta$ starting at a point $p$ and a horizontal distribution $\mathcal{H}$ for $f$, $V_0^\mathcal{H}(\gamma_p)$ is the subset of the fiber $f^{-1}(p)$ consisting of points that flow to a critical point of $f$. The vanishing cycles in Figure \ref{crit} are examples. Other examples include a pair of points, one in each disk in a fiber $D^2\sqcup D^2$ in Figure \ref{fold} and a pair of points in the disk fiber in Figure \ref{cusp}, all of which flow to the unique critical arc in each figure. For this reason, the letter $V$ stands for \emph{vanishing set} in the following definition.

\begin{df}[Disjointness condition]\label{dsjtcond}Let $\mathcal{H}$ denote a horizontal distribution for the Morse 2-function $f_0$ over a neighborhood of $\Delta$, obtained by taking the orthogonal complement $\ker(df_0)^\perp$ with respect to some Riemannian metric on $f_0^{-1}(\Delta)$. Consider the union over all regular values $p$ of $f_0$, 
\[V=\bigcup\limits_p V_0^\mathcal{H}(\gamma^{up}_p)\cup V_0^\mathcal{H}(\gamma^{down}_p),\] 
where the smooth embeddings $\gamma^{up}_p,\gamma^{down}_p\co[0,1]\to\Delta$ travel from $p$ vertically up or down, respectively, to the two boundary arcs of $\Delta$. Let $Z$ be the points in  $V$ that run into $\varphi_0$ and let $Y$ be the points in $V$ that run into $\crit(f_0)\setminus \varphi_0$. The disjointness condition is that there exists an $\mathcal{H}$ such that $Y$ and $Z$ are bounded away from each other in the preimage of $\Delta$.\end{df}
In this definition, it is acceptable for the terminal point of $\gamma^{down}_p$ to be a critical value in $f(\varphi_0)$ because we only consider the vanishing set $V_0$ of points in the regular fiber $f_0^{-1}(p)$ that run into critical points (not the ones that emerge from critical points). Similarly, there may be arcs $\gamma^{up}_p$ whose terminal point is a critical value. Vanishing sets appear in the definition instead of ascending and descending manifolds because they are defined. For example, suppose $\gamma^{up}_p$ is tangent to the image of a fold arc. In that case, the restriction of $f_0$ to the preimage of $\gamma^{up}_p$ is not a Morse function, so it is not appropriate to speak of the descending manifold of the critical point of $f_0$ above that tangency. On the other hand, the vanishing set is defined in \cite{BH} and has the same properties related to pushing around critical values. Finally, the general process of depicting a vanishing cycle in $f^{-1}(p)$ using a reference path and horizontal distribution will be called \emph{measuring a vanishing cycle}.

\begin{rmk}\label{deffromH}A disk $\Delta$ and a horizontal distribution $\mathcal{H}$ satisfying the disjointness condition over $\Delta$ can be used to specify a deformation suggested by $\Delta$. Let $N$ be a neighborhood of the vanishing set of $\varphi_0$ over a neighborhood $\nu\Delta$ of $\Delta$ (note $N$ is a ball, as discussed in \cite[Section~2.3.1]{W1}). Choose local coordinates near $N$ so that $f_0|_N$ is given by the local model for folds. Then $x_1$ parameterizes $\varphi_0$, and the level sets for suitably parameterized vertical arcs foliating $\nu\Delta$ are the level sets of $-\sum\limits_{i=2}^{k+1}x_i^2+\sum\limits_{i=k+2}^nx_i^2$. Such local coordinates exist because $N$ consists of regular points, with the exception of $\varphi_0$, by the disjointness condition. Now, restricting to $N$, Figure~\ref{isot} is merely the image of the local model for folds, and there is certainly a deformation of the local model $N\to\nu\Delta$, fixing the map near $\partial\overline{N}$, realizing the suggested movement in $S^2$. Call it $f^N_t$, $t\in[0,1]$. Then the deformation is given by \[f_t(x)=\begin{cases}f^N_t(x)&x\in N\\f_0(x)&x\notin N.\end{cases}\]\end{rmk}

\begin{rmk}\label{vertarcs}One further observation is that $Z\cap Y$ always maps to a family of vertical arcs connecting the two sides of $\Delta$, because the union of paths traced by any point in $f_0^{-1}(p)$ above $\gamma_p^{up}$ and its reverse coincides with that of $\gamma_p^{down}$, and its reverse.\end{rmk}

\begin{thm}\label{isotthm}The disjointness condition of Definition~\ref{dsjtcond} is equivalent to the existence of a deformation as suggested by Figure~\ref{isot}.\end{thm}
\begin{proof}The disjointness condition is merely a restatement of the requirement that the relevant critical sets of the restrictions of $f_0$ to the vertical arcs in $\Delta$ simultaneously satisfy the same disjointness condition for switching critical points of smooth real-valued functions.\end{proof}
Here follow some examples of how to use Theorem~\ref{isotthm} to verify a proposed movement of critical arcs in the base when $n=4$. Note that indefinite folds in this case have a higher-genus and lower-genus side because the regular fibers of such maps are closed, orientable surfaces. Existence results for various types of what are called \emph{Reidemeister-2 fold crossings} in \cite{GK3} and \cite{W2} are collected into the following proposition.
\begin{prop}[Existence of Reidemeister-2 fold crossings]\label{r2prop}\  \begin{enumerate}
\item\hypertarget{r2prop(1)} There exists a deformation as suggested in Figure~\ref{finger} if one of the following two conditions occur.\begin{enumerate}
	\item[(\hyperlink{r2propitem1aproof}{a})]\hypertarget{r2propitem1a} The point $p$ is on the lower-genus side of at least one of the pictured fold arcs.
	\item[(\hyperlink{r2propitem1bproof}{b})]\hypertarget{r2propitem1b} The point $p$ is on the higher-genus side of both fold arcs, and there exist embedded reference paths from $p$ to each fold arc within $\Delta$ for which the measured vanishing cycles in $f_0^{-1}(p)$ are disjoint.\end{enumerate}
\item[(\hyperlink{r2propitem2proof}{2})]\hypertarget{r2prop(2)}In Figure~\ref{r2propfig}, suppose the fiber genus of each fiber component above $p$ is at least 2, and that $p$ lies on the higher genus side of exactly one of the fold arcs, say $\varphi_0$. Then the suggested $R_2$ deformation exists if  and only if the vanishing cycles at $\Delta(\pm1)$ match, as measured with reference paths from $\Delta(i)$.
\item[(\hyperlink{r2propitem3proof}{3})]\hypertarget{r2prop(3)}There always exists an $R_2$ deformation as suggested by Figure \ref{r2propfig} when $p$ is on the lower-genus side of both critical arcs, even when those critical arcs contain cusp points.
\item[(\hyperlink{r2propitem4proof}{4})]\hypertarget{r2prop(4)}In Figure~\ref{r2propfig}, suppose $p$ lies on the higher-genus side of both fold arcs, all regular fibers are connected, and the fiber above $p$ has genus at least 2. Then there exists an $R_2$ deformation as suggested by the figure.
\end{enumerate}\end{prop}
\begin{figure}
	\centering{
	\labellist
	\small\hair 2pt
	\pinlabel $p$ at 95 33
	\endlabellist
	\includegraphics{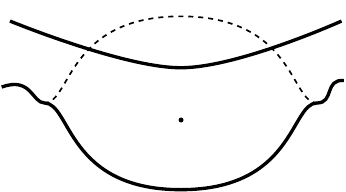}}
	\caption{Preparing to push a bit of a fold arc past another, in what is commonly called a \emph{finger move}, for Proposition~\ref{r2prop}(\protect\hyperlink{r2prop(1)}{1}).}\label{finger}
\end{figure}
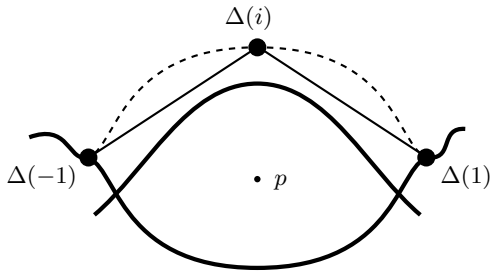
\begin{figure}
	\centering{
	\labellist
	\small\hair 2pt
	\pinlabel $p$ at 95 33
	\pinlabel $\Delta(i)$ at 83 95
	\pinlabel $\Delta(-1)$ at 5 35
	\pinlabel $\Delta(1)$ at 165 35
	\endlabellist
	\includegraphics{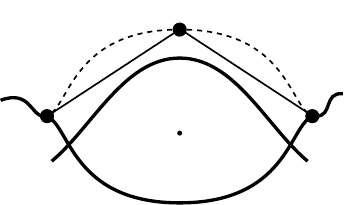}}
	\caption{The disk $\Delta$ with lower fold arc $f_0(\varphi_0)$ in preparation for an $R_2$ deformation in items (\protect\hyperlink{r2prop(2)}{2})-(\protect\hyperlink{r2prop(4)}{4})	 of Proposition~\ref{r2prop}.}\label{r2propfig}
	\end{figure}
\begin{proof}
The typical approach for a statement like Assertion~(\hyperlink{r2prop(1)}{1}) will be to show that there is a choice of horizontal distribution such that $Z\cap Y=\emptyset$. First, note that $p$ must either lie on the lower-genus side of at least one fold arc which we will call $A$ (case (a)), or on the higher-genus side of both (case (b)). In either case, the deformation exists if and only if there is a vertical arc $v\subset\Delta$ connecting the two fold arcs that is disjoint from $f_0(Z\cap Y)$, because in that case one could shrink $\Delta$ so that its interior is a tubular neighborhood of $v$ whose preimage satisfies the disjointness condition.

\hypertarget{r2propitem1aproof}For case (\hyperlink{r2prop(1)}{1a}), $Z\cap Y$ is 1-dimensional if $p$ lies on the lower-genus side of exactly one of the fold arcs, and by Remark~\ref{vertarcs} it maps to a disjoint union of vertical arcs in $\Delta$. If $p$ lies on the lower-genus side of both of the fold arcs then $Z\cap Y$ is $0$-dimensional, the intersection of two surfaces in $M_I$. However, Remark~\ref{vertarcs} implies the dimension of $Z\cap Y$ is at least 1, for any $\mathcal{H}$. For this reason, $Z\cap Y$ must be empty in this case. 

\begin{rmk}\label{digr}If the previous paragraph sounds strange, here is a digression, giving another way to explain its conclusion: Any arc in $Z\cap Y$ is simultaneously the cocore of a three-dimensional 2-handle and the core of a 3-dimensional 1-handle, which is not only non-generic in the space of Morse functions because it is not Morse-Smale, but is also non-generic in the space of one-parameter families of Morse functions (parameterized by the horizontal direction $s$ in Figure~\ref{finger}): Over any horizontal line $\ell(s)$, the pairs of points that run into the top and bottom fold arcs trace paths in the 3-manifold $($fiber$)\times\ell$, and an arc in $Z\cap Y$ is an arc of intersections between these paths. Given such an arc $\gamma\subset Z\cap Y$, one merely perturbs $\mathcal{H}$ so that the $s$-value at which one pair of paths crosses $\gamma$ is different than the $s$-value for the other, and this can be achieved all along $\gamma$ by performing essentially the same perturbation crossed with $\gamma$. This concludes the digression.\end{rmk}

In either case, the dimension of $Z\cap Y$ is no greater than 1, so there is at least one vertical arc in $\Delta$ whose preimage in $Z$ is disjoint from $Y$, and so this arc has a neighborhood whose preimage in $Z$ is disjoint from $Y$, and we perform the finger move in this neighborhood.

\hypertarget{r2propitem1bproof}For case (\hyperlink{r2prop(1)}{1b}), the assumption is that $Z\cap Y=\emptyset$ in the fiber above $p$. In the language of handlebodies, we have a three-dimensional 1-handle attachment followed by a 2-handle attached along a circle that can up to isotopy is disjoint from the belt circle of the 1-handle. In this case, it is also known that the order of handle attachment can be switched by a homotopy of the corresponding Morse function, then un-switched, which similarly gives the suggested $R_2$ deformation.

\hypertarget{r2propitem2proof}For the ($\Leftarrow$) part of Assertion~(\hyperlink{r2prop(2)}{2}), assume the fiber above $p$ is connected, since the result obviously holds if the vanishing cycles of the two fold arcs, as measured from $\Delta(i)$, lie in different fiber components. Denote by $F$ the generic fiber over $\Delta(i)$ and $F'$ the fiber over $p$. Also, denote by $q,q'$ the pair of points in $F'$ corresponding to the fold arc not containing $\varphi_0$, again assumed to be stationary. Because of the matching vanishing cycles, it makes sense to talk about a single vanishing cycle $v$ of the fold arc containing $\varphi_0$ that lives in $F$, measured near either end of $\varphi$ (and also in $F'$ for those points actually in $\varphi$). Choosing small $\epsilon>0$, the representatives of $v$ as measured at $\Delta(\pm1\mp\epsilon)$ are isotopic in $F'\setminus \{q,q'\}$ because they are isotopic in $F$. The surface $F'\setminus\{q,q'\}$ has free fundamental group, in which elements whose representatives differ by crossing $q,q'$ are distinct. This implies that each crossing of $v$ over $q$ or $q'$ (that is, each point in $Z\cap Y$, which is generically a  1-submanifold of $M$) is accompanied by a canceling crossing (paired, say, by letting the endpoints of reference paths travel along horizontal lines across the central lune to get one-parameter families of points and vanishing cycles in $F'$). This allows the disjointness condition to be satisfied by a homotopy of $\mathcal{H}$ that cancels the crossings by identifying vertical arcs in $\Delta$ in pairs. The ($\Rightarrow$) part of Assertion~(\hyperlink{r2prop(2)}{2}) is clear; see Example~\ref{r2prop(3)nonex} for a few examples.

\hypertarget{r2propitem3proof}For Assertion (\hyperlink{r2prop(3)}{3}), after moving all cusps into the higher-genus sides of the two folds, the proof of the analogous case of Assertion (\hyperlink{r2prop(1)}{1a}) (the last two sentences above Remark~\ref{digr}) goes through word-for-word.

\hypertarget{r2propitem4proof}Now for Assertion (\hyperlink{r2prop(4)}{4}). If $Z$ and $Y$ have no intersections above the lune $B\subset\Delta$ bounded by the fold arcs, then it is possible to choose $\Delta$ small enough to make $Z\cap Y=\emptyset$. Assume without loss of generality that $Z$ and $Y$ intersect transversely, and let $p\in L$. Since all fibers are connected, the two vanishing cycles as measured at $p$ using the reference paths $\gamma^{up}_p,\gamma^{down}_p$ are homologically distinct and disjoint for all $p$ sufficiently close to the fold crossings. In this situation, any point of intersection between $Z$ and $Y$ is reflected by an intersection between the vanishing cycles in the fiber using vertical reference paths, and this point is one of a pair that forms the corners of a family of lunes, one in in each fiber, each of whose boundary is a pair of arcs, one in each of the two vanishing cycles. By a homotopy of $\mathcal{H}$, these lunes can be simultaneously contracted, eliminating the intersections of $Z$ and $Y$ from $L$. Now the proposition follows from Theorem~\ref{isotthm}.
\end{proof}
Here is one more existence result for finger moves involving cusps.

\begin{figure}
	\centering
	\includegraphics{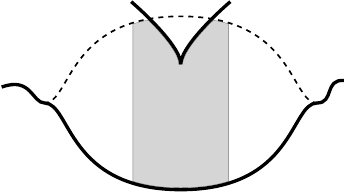}
	\caption{Pushing a fold across a cusp. Here, $\Delta$ could lie on the higher or lower genus side of $f_0(\varphi_0)$.}\label{cuspfold}
\end{figure}
\begin{prop}\label{cuspfoldprop}There always exists a \emph{cusp-fold crossing} deformation as suggested by Figure~\ref{cuspfold}.\end{prop}
\begin{proof}The proof of this result is just the same as that of  Proposition~\ref{r2prop}(\hyperlink{r2prop(1)}{1a}), except that $Z$, instead of being a pair of embedded disks given by the fiberwise feet of a 3-dimensional 1-handle, is an immersed pair of disks whose immersion locus is an arc of double points limiting to the cusp, and the argument applies after perturbing $\mathcal{H}$ to eliminate intersections between $Z$ and $Y$ along this arc similarly to Remark~\ref{digr}.\end{proof}

\begin{ex}\label{r2prop(3)nonex}Here are a few examples illustrating subtleties surrounding Proposition~\ref{r2prop}\hyperlink{r2prop(2)}{(2)}. Suppose that above $\Delta^{-1}(\{\Rp z=0\})$ there is an arc of intersection between $Z$ and $Y$, where the vanishing cycle of $\varphi_0$ crosses once over one of the points that runs into the other fold arc, traveling along the reference path $\{\Rp z=0\}$. Then up to isotopy the two vanishing cycles mentioned in Proposition~\ref{r2prop}\hyperlink{r2prop(2)}{(2)} differ by sliding over a disk bounded by the other fold arc's vanishing cycle, so that the $R_2$ deformation would result in the particular impossibility that validates the ($\Rightarrow$) part of the statement: $\varphi_0$ would be free of intersection points with other folds and have vanishing cycles at either end that are not isotopic in the fiber above $\Delta(i)$.

Another more subtle problem occurs when the regular fiber over the lune is a torus or a sphere. In this case, there can be many vertical arcs in $f_0(Z\cap Y)\cap\Delta$ in which the vanishing cycle crosses one of the points with the same orientation (think of a meridian of a torus traveling around and around the torus, repeatedly crossing some marked point). Unlike the previous example, the vanishing cycles of $\varphi_0$ at $\Delta(\pm1)$ can match, and the base diagram one might expect from such a move might not have obvious inconsistencies, but the proposed movement of fold arcs in the base is still not valid by Theorem~\ref{isotthm}.

In all these cases, $Z\cap Y$ is generically 1-dimensional: over every vertical arc (oriented downward) in $\Delta$ the Morse function corresponds to a pair of consecutive 3-dimensional 2-handle attachments, and every point of $Z\cap Y$ is an intersection between the cocore of the first 2-handle (1-dimensional in each arc, sweeping out a 2-dimensional subset of $M_0$) and the core of the second (2-dimensional in each arc, sweeping out a 3-dimensional subset of $M_0$). Certainly there is no way to switch the order of handle attachments while keeping the intersection between core and cocore inside the lune, and (as the first example shows) there is no general way to move a vertical arc of intersections out of $\Delta$ by a homotopy.\end{ex}

\begin{figure}[h]
	\centering{
		\labellist
		\small\hair 2pt
		\pinlabel $B$ at 10 10
		\pinlabel $A$ at 85 10
		\pinlabel $C$ at 165 10
		\endlabellist
		\includegraphics{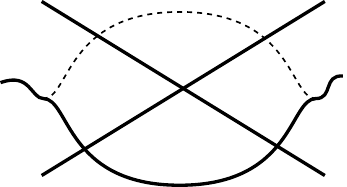}}
	\caption{A Reidemeister-3 fold crossing.}\label{triangle}
\end{figure}
\begin{prop}[Existence of Reidemeister-3 fold crossings]\label{r3prop}A Reidemeister-3 fold crossing exists if there are no cusps in the triangle and the minimal fiber genus in a neighborhood of the triangle is at least 2.\end{prop}
\begin{proof}Figure~\ref{triangle} depicts the closure $\Delta$ of a neighborhood of a triangle of fold arcs. Foliate $\Delta$ by a family of vertical paths, each of which can be considered the target of a member of a one-parameter family of Morse functions. There are four cases, according to the number $n$ of fold arcs that have the central triangle on their lower-genus sides. In all cases, the restriction of the map to the preimage of each leaf is a Morse function with three critical points and two or three critical values. The depicted Reidemeister-3 fold crossing can be thought of as affecting the vertical order of the critical values of these Morse functions in various ways. For any Reidemeister-3 fold crossing, we let the disk $\Delta$ be the closure of a regular neighborhood of the triangle formed by the three fold arcs, with $f_0(\varphi_0)$ coinciding with $A$. Since the fiber is connected with genus at least 2 over every regular value, the vanishing cycles of any one of the three folds as measured by different reference paths from the some arbitrarily chosen point can be assumed to be equal by appropriate choice of horizontal distribution, not just isotopic. Example~\ref{r3ex} shows why this is a necessary assumption.

The case $n=3$ can be satisfied as in Proposition~\ref{r2prop}(\hyperlink{r2prop(3)}{3}): though they would be intersections between two surfaces in a 4-manifold, intersections between the ascending manifold of $A$ and the descending manifold of $B$ or $C$ are not generic.

For $n=2$, without loss of generality suppose the triangle is on the higher-genus side of $B$ and the lower genus sides of $A$ and $C$. As in the paragraph above Proposition~\ref{r3prop}, generically the intersection between the ascending manifold of $A$ and the descending manifold of $B$ maps to a collection of vertical line segments that are not clearly removable. However, it is possible to move such a line segment to the left across $C$, out of the triangle, to lie in the higher-genus side of $C$. Intersections coming from $A$ and $C$ are ruled out as in the case $n=3$.

For $n=1$, without loss of generality suppose the triangle is on the higher-genus sides of $B$ and $C$ and the on the lower genus side of $A$. As in the case $n=2$, $Z\cap Y$ is a finite number of line segments connecting $A$ to $B$ (or $A$ to $C$). Since $B$ and $C$ intersect, their vanishing cycles $b_p,c_p$ (as measured from any reference point $p$ in their higher-genus sides) are isotopic to disjoint circles. For this reason, for every $p$, every transverse intersection $b_p\cap c_p$ is one of a pair of corners of a lune in $F_p$, with one side lying in $b_p$ and the other side lying in $c_p$. By smoothness of the fibration map, these lunes can be chosen to vary smoothly with $p$. For all $p$ on the higher genus side of both $B$ and $C$, simultaneously collapse the lunes by an isotopy sending the $b_p$ side to the $c_p$ side of each. What results in each fiber is a collection of arcs (perhaps even all of $B$ and $C$ if they are isotopic) on which the intersections of the fiber with $B$ and $C$ coincide. These arcs vary smoothly with $p$ and it is possible to perturb $b_p$ smoothly over all $p$ to replace each arc with two parallel arcs, one for $b_p$ and one for $c_p$. Now points in $Z\cap Y$ coming from $A$ and $B$ are necessarily disjoint from $c_p$ for all $p$ in the triangle, so they can be pushed out of the triangle across $C$, and those points in $Z\cap Y$ coming from $A$ and $C$ can be pushed out across $B$.

The $n=0$ case, which is not used in this paper because it is strictly genus-decreasing, is included for completeness. Here, $Z\cap Y$ is generically a surface in $M$ mapping to a codimension 0 subset of $\Delta$, naturally interpreted as a collection of intersection points between vanishing cycles. Let $a_p,b_p,c_p$ be the three vanishing cycles for $A,B,C$, respectively, measured from a reference fiber $F_p$ over a point $p$ in the triangle. All three fold arcs intersect pairwise, so that $a_p,b_p,c_p$ are isotopic to pairwise disjoint circles in $F$. As in the $n=1$ case, it is possible to arrange for $b_p$ and $c_p$ to be disjoint in every fiber on the higher-genus sides of $B$ and $C$. In particular, this eliminates any triple intersections in $Z\cap Y$, so that points in that set coming from $A$ and $C$ are disjoint from $B$, so they can be pushed across $B$, and similarly those coming from $A$ and $B$ can be pushed across $C$. The three remaining bits at the angles opposite the corners of the triangle can be assumed free of intersections by choosing $\Delta$ small enough.
\end{proof}
\begin{ex}\label{r3ex}This example, adapted from one due to an anonymous referee for \cite{W2}, shows why the genus assumption in Proposition~\ref{r3prop} may be necessary. Suppose $n=0$, so that the triangle is on the higher-genus side of all three fold arcs, the regular fiber over any point inside the triangle has genus 1, and that the vanishing cycles for all three fold arcs are isotopic in the central fiber to a meridian. Choose the reference fiber that lies over the point $p$ at the center of the triangle, which we suppose is equilateral. For a short interval of radial reference paths from $p$ to points in the interior of one of the sides of the triangle, suppose the vanishing cycle travels once around the torus. In this situation, it seems the disjointness condition between $Z$ and $Y$ cannot be arranged by pushing intersections out of the triangle.\end{ex}

\begin{prop}[Genus-increasing movements of an isolated fold arc]\label{giprop}There exists a sequence of two-parameter crossings as suggested by Figure~\ref{isot} when the lune is on the lower-genus side of $f(\varphi_0)$, $f(\varphi_0)$ does not contain any crossings in the critical image, and all components of all regular fibers inside the lune have genus at least 2.\end{prop}
\begin{proof}The proposed movement exists by the genus-increasing 2-parameter crossing propositions, namely Proposition~\ref{r2prop}(\hyperlink{r2propitem1a}{1a},\hyperlink{r2prop(2)}{2},\hyperlink{r2prop(3)}{3}) for $R_2$ deformations, Proposition~\ref{cuspfoldprop} for cusp-fold crossings, and $n\neq0$ cases of Proposition~\ref{r3prop} for Reidemeister-3 fold crossings. The hypotheses of these results are satisfied automatically, except those of Proposition~\ref{r3prop} (simply observe that any cusps in the triangle can be first pushed into the higher-genus side of $\varphi_0$, which by hypothesis has no cusps) and Proposition~\ref{r2prop}(\hyperlink{r2prop(2)}{2}).

In the context of Proposition~\ref{r2prop}(\hyperlink{r2prop(2)}{2}) and in Figure~\ref{r2propfig}, the fold arc $\varphi_0$ for the present proposition (call it $\varphi^{\text{present}}$) would be the upper fold arc, while the $\varphi_0$ mentioned in Proposition~\ref{r2prop}(\hyperlink{r2prop(2)}{2}) (call it $\varphi^{(2)}$) would be the lower fold arc. One may choose the movement of $\varphi^{\text{present}}$ as beginning with a pair of finger moves that exist by Proposition~\ref{r2prop}(\hyperlink{r2propitem1a}{1a}), so that the first pair of intersections as in Proposition~\ref{r2prop}(\hyperlink{r2prop(2)}{2}) to form during the movement comes from the two finger moves. By appropriately choosing a path in $M_0$ for the tip of the second finger move to trace, one may choose the location of the two points that become identified in the fiber when $\varphi^{\text{present}}$ passes by; in particular, near the crossings of Figure~\ref{r2propfig}. In this way, it is possible to arrange for the two vanishing cycles of $\varphi^{(2)}$ to match, allowing an application of Proposition~\ref{r2prop}(\hyperlink{r2prop(2)}{2}) to unify the two fingers into one thick finger. The movement continues with another finger move, chosen appropriately as before, and so on until the proposed movement is finished.\end{proof}

\section{Simplifying a Morse 2-function to the 2-sphere}\label{simpsection}
This section gives a new proof of the existence result \cite[Corollary 2]{W1}. The argument in that paper begins with the main theorem of the 41-page paper \cite{GK1}, an existence result that in turn relies on the Giroux correspondence of contact geometry \cite{Et} and Eliashberg's classification of overtwisted contact structures on 3-manifolds \cite{E}. Recently, there appeared another algorithmic existence result in \cite{BS}. This argument is significantly shorter than both, relying instead entirely on Section \ref{thmandapps}.

It is a classical result that any continuous map from a smooth, closed oriented 4-manifold $M$ to the 2-sphere can be perturbed to be a smooth map whose critical locus consists of a smooth disjoint union of circles in $M$, stratified into a collection of definite and indefinite fold arcs which meet at cusp points. The main result of \cite{S} states that the definite locus may be eliminated by homotopy, resulting in what was named an \emph{indefinite Morse 2-function} in \cite{GK3}. This section is about indefinite Morse 2-functions from a closed, orientable 4-manifold $M$ to $S^2$. The critical locus of such a map may have more than one path component and its image may have transverse double points.

\begin{df}An \emph{equivalence} $f_t$ of indefinite Morse 2-functions $f_0$, $f_1$ is a deformation which begins with $f_0$, ends with $f_1$, and consists of a sequence of the moves of \cite{Le} (birth, merge, flip) and the various indefinite 2-parameter crossings of \cite{GK2}.\end{df} 
\begin{thm}\label{T}For any indefinite Morse 2-function $f\co M\to S^2$, there is an algorithm that produces an equivalence between $f$ and an indefinite Morse 2-function which is injective on its critical locus, which is connected.\end{thm}
\begin{proof}The algorithm is to eliminate the crossings between indefinite fold arcs (Lemma \ref{clemma}), then unify the resulting collection of circles without introducing further crossings (Lemma \ref{ulemma}). Figure~\ref{c} gives the recipe for Lemma \ref{clemma}, and its explanation is the main content of the section.\end{proof}
The map $f$ in Theorem \ref{T} is known as a \emph{simple wrinkled fibration} in \cite{BH} or a \emph{simplified purely wrinkled fibration} in \cite{Le,W1}, and while the term \emph{Morse 2-function} due to Gay and Kirby has become the standard term for such maps when $f|_{\crit f}$ is allowed to have crossings and $\crit f$ is allowed to have multiple components, this author refers to the maps in the above definition as \emph{crown maps}, a term due to Rob Kirby.
\begin{lemma}[Resolving crossings]\label{clemma}The movements suggested by Figure~\ref{c} come from an equivalence.\end{lemma}
\begin{proof}The following argument uses Proposition \ref{giprop}, so the first step is to make sure the assumption on genus in that proposition is satisfied, using what will be called \emph{genus inflation} in Remark \ref{genusinflationremark} below. Genus inflation begins by first performing birth moves, then performing finger moves (Proposition \ref{r2prop}(\hyperlink{r2prop(1)}{1a})), cusp-fold crossings (Proposition \ref{cuspfoldprop}), and Reidemeister-3 fold crossings (Proposition \ref{r3prop}) so that the lunes in $S^2$ introduced by the birth moves (that is, the disks in $S^2$ bounded by the images of the critical circles) have sufficient overlap to increase the genus of all fiber components to at least 2. The use of Proposition \ref{r3prop} in that process has its own fiber genus requirement, which can be satisfied by the following construction: For a chosen fiber component $F$ of the fibration map $f$, choose an open 2-dimensional disk $D\subset M$ of regular points passing through $F$ and such that	 $f|_D$ is a diffeomorphism onto a neighborhood of $\Delta$ (such $D$ exists because any transverse intersection between a 2-dimensional disk and $\crit f$ is empty). Then there is a neighborhood $\nu D$ of $D$ on which $f$ is a submersion, so it is possible to perform a birth move on $f|_{\nu D}$ whose critical circle has image approximating $f(\partial D)$, increasing the fiber genus in $\Delta$ by a self-connect sum at each point $D\cap f^{-1}(x)$, $x\in\Delta$.

Figure~\ref{c1} depicts a fold crossing, with arrows as in \cite{GK3} indicating that the restriction of $f$ to the preimage of each indicated transverse oriented arc is a Morse function with a single index 2 critical point. In a move inspired by \cite[Figure 10]{Le}, perform two flips to obtain \ref{c2}, followed by a cusp merge that exists as discussed in \cite{BH,Le} to obtain \ref{c3}. Figure~\ref{c4} is an application of Proposition \ref{giprop} according to the dotted line in Figure \ref{c3}.

\begin{figure}[h]
    \centering
    \begin{subfigure}{2.5cm}
	\labellist
	\small\hair 2pt
	\pinlabel $2$ at 33 20
	\pinlabel $2$ at 47 20
	\endlabellist
        \includegraphics{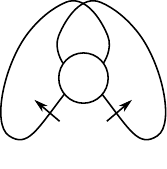}
        \caption{A crossing.}
        \label{c1}
    \end{subfigure}
    \quad
    \begin{subfigure}{2.5cm}
        \includegraphics{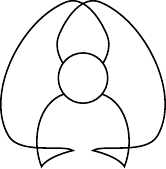}
        \caption{Two flips.}
        \label{c2}
    \end{subfigure}
    \quad
    \begin{subfigure}{2.5cm}
	\labellist
	\small\hair 2pt
	\pinlabel $\beta$ at 27 19
	\pinlabel $a$ at 19 32
	\pinlabel $b$ at 24 68
	\pinlabel $c$ at 60 32
	\endlabellist
        \includegraphics{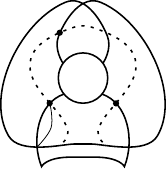}
        \caption{Cusp merge.}
        \label{c3}
    \end{subfigure}
    \quad
    \begin{subfigure}{2.5cm}
        \includegraphics{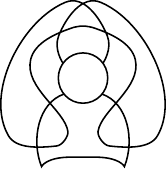}
        \caption{Prop. \ref{giprop}}
        \label{c4}
    \end{subfigure}\\ \vspace{.5cm}
    \begin{subfigure}{2.5cm}
        \includegraphics{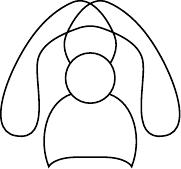}
        \caption{Prop. \ref{r2prop}(\protect\hyperlink{r2prop(2)}{2}).}
        \label{c5}
    \end{subfigure}
    \quad
    \begin{subfigure}{2.5cm}
        \includegraphics{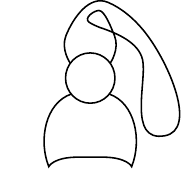}
        \caption{Prop. \ref{r2prop}(\protect\hyperlink{r2prop(2)}{2}).}
        \label{c6}
    \end{subfigure}
    \quad
    \begin{subfigure}{2.5cm}
        \includegraphics{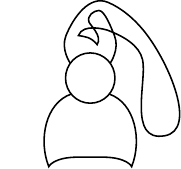}
        \caption{Flip.}
        \label{c7}
    \end{subfigure}
    \quad
    \begin{subfigure}{2.5cm}
        \includegraphics{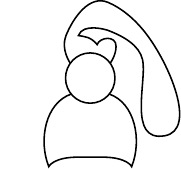}
        \caption{Prop. \ref{r2prop}(\protect\hyperlink{r2prop(3)}{3}).}
        \label{c8}
    \end{subfigure}
    \caption{Resolving a crossing between the images of two fold arcs. These are base diagrams, where $S^2$ is punctured at a point near the crossing and the rest of the critical image is in the circle in each figure. The dots and the curve $\beta$ in \ref{c4} correspond to parts of $\partial\Delta$ along which the suggested movement is required to satisfy an extra condition.}\label{c}
\end{figure}

In this paragraph, let $f_0$ denote the fibration depicted in Figure \ref{c3} and let $\varphi_0$ denote the fold arc forming part of 
$\partial\Delta$, with the rest of $\partial\Delta$ depicted by the dotted line in that figure. As with any movement of a fold arc, the 
movement depicted in Figures \ref{c3}-\ref{c4} is encoded by the ascending manifold $Z$ of $\varphi_0$ over $\Delta$. This movement is  
genus-increasing, so $Z$ intersects each fiber over $\Delta\setminus\varphi_0$ in two points. As discussed in the previous section, a change  
in $Z$ is equivalent to a change in the horizontal distribution $\mathcal{H}$ used to define the deformation which moves $\varphi_0$ (see also \cite[Lemma 4.6]{BH}). With  
this in mind, the claim is that it is possible to choose $Z$ for the movement depicted in Figures \ref{c3}-\ref{c4} so that the matching  
requirement of Proposition \ref{r2prop}(\hyperlink{r2prop(2)}{2}) is satisfied for the movements in Figures \ref{c4}-\ref{c6}. To prove this
claim, consider the lune containing the smooth injective curve $\beta\co[0,1]\to S^2$ in Figure \ref{c3}. The left side is a fold arc and the 
right side comprises two arcs, one in $\varphi_0$ and the other in $\partial\Delta$. Foliate this lune by line segments $\ell_t(s)\co[0,1]\to 
S^2$ transverse to $\beta$, where $\ell_t\cap\beta=\beta(t)$. Parameterize $\ell_t$ and $\beta$ so that $\ell_t(0)\in\partial\Delta$ for $t\geq1/2$. 
In the fiber $f_0^{-1}(\beta(1/2))$, let $C_t$ be the vanishing cycle measured from $\ell_t(1)$, using a reference path that travels along 
$\beta$ from $\beta(1/2)$ to $\beta(t)$ and then along $\ell_t$. Let $(z_1,z_2)_t$ be the pair of points that run into either $Z$ or the 
critical point mapping to $\ell_t(0)$ using the analogous reference path. This data specifies a family $\mathcal{F}_t=(C,z_1,z_2)_t$, 
$t\in[0,1]$. In other words, $\mathcal{F}_t$ is a circle and two points, with each one moving around the regular fiber by isotopy, perhaps 
intersecting each other, with the paths taken by $z_1$ and $z_2$ for $t>1/2$ dictated entirely by the placement of $Z$. Now isotope $Z$ in fibers near $f_0^{-1}(a)$, changing $\mathcal{F}_t$ for $t\in(1/2,1)$, so that 
$\mathcal{F}_t$ is isotopic to a family $\mathcal{F}'_t$ satisfying $\mathcal{F}'_t=\mathcal{F}'_{1-t}$. With this achieved, the required 
matching condition is satisfied, and there are analogous modifications near the points $x$ and $y$ for the other two applications of 
Proposition \ref{r2prop}(\hyperlink{r2prop(2)}{2}) depicted in \ref{c4}-\ref{c6}. 

A homotopy inducing the rest of the movements (\ref{c6}-\ref{c8}) exists by the existence of the flipping move and Proposition \ref{r2prop}(\hyperlink{r2prop(3)}{3}).\end{proof}
\begin{lemma}\label{ulemma}The critical circles from Lemma~\ref{clemma} can be unified without introducing new intersection points.\end{lemma}
\begin{proof}Given a fibration that is injective on its critical locus and has more than one critical circle, there is a pair of circles $S_1$ and $S_2$ of critical values bounding disjoint embedded open disks $D_1$, $D_2$, respectively, with a smooth embedded path $\beta$ of regular values connecting $p_1\in S_1$ to $p_2\in S_2$. 

If $\beta$ lies on the higher-genus side of each circle, then the circles can be unified by two flips and a cusp merge as in Figures \ref{c1}-\ref{c3}. Each disk $D_i$ now has one crossing in its boundary coming from the flips. These can be eliminated by a flip and an application of Lemma \ref{giprop} followed by \ref{r2prop}(\hyperlink{r2prop(3)}{3}), as in Figures \ref{c5}-\ref{c8} ($D_i$ corresponds to the lune in Figure \ref{c7} which is gone in Figure \ref{c8}).

If $\beta$ lies on the lower-genus side of $S_i$, pick a small neighborhood $\nu p_i\subset S^2$, define $\Delta=S^2\setminus(\nu p_i\cup D_i)$ and apply Proposition \ref{giprop}. This causes $\beta$ to lie on the higher-genus side of $S_i$. In this way, reduce to the previous case.\end{proof}
\begin{rmk}\label{genusinflationremark}It is not necessary for this paper, but it seems likely that the modifications subsequent to genus inflation up to this remark could be performed using homotopies whose support is disjoint from that of the genus inflation. In that case, the genus inflation would have been unnecessary and Example \ref{r3ex} would present a non-issue.\end{rmk}
\section{Refined uniqueness results for crown diagrams}\label{refineduniqnesssection}
This section uses a number of ideas and results from \cite{W2}, which in turn uses results from Section \ref{thmandapps}. It also relies on the excellent papers of Behrens and Hayano \cite{H,BH} as a starting point for its main results; some familiarity with these three papers will be necessary to follow its arguments. Terminology from \cite{GK3} is used as much as possible, though the content of that paper is not required for this section.

Continuing the theme of this paper, this section is about modifying certain fibration maps using existence results for homotopies of those maps. These modifications are used to better understand how much variety is available to the collection of \emph{crown diagrams} coming from a fixed homotopy class of maps $M\to S^2$. 
\begin{df}\label{sddef}Suppose $f\co M\to S^2$ is an indefinite Morse 2-function which is injective on its critical set, and its critical set is connected. Choose a reference fiber $\Sigma$ on the higher genus side and let $\Gamma=(\gamma_i)_{i\in\Z/k\Z}$ be the cyclically indexed sequence of vanishing cycles up to isotopy measured from $\Sigma$. If the genus of $\Sigma$ is at least 3, then the pair $(\Sigma,\Gamma)$ is called a \emph{crown diagram} of $M$.\end{df}This section freely uses language from \cite{BH,GK2,W2}. It is known that every smooth closed oriented 4-manifold has crown diagrams coming from maps in every homotopy class of maps to $S^2$ \cite{W1}, and that any crown diagram specifies $M$ up to orientation-preserving diffeomorphism \cite{W1}. The uniqueness result this section aims to refine is as follows.
\begin{thm}[\cite{W2}]\label{p2thm}Suppose $\alpha_0,\alpha_1\co M\to S^2$ are two homotopic crown maps with crown diagrams $(\Sigma^i,\Gamma^i)$, $i=0,1$. Then there is a finite sequence of crown diagrams, beginning with $(\Sigma^0,\Gamma^0)$ and ending with $(\Sigma^1,\Gamma^1)$, obtained by performing stabilizations, handleslides, shifts, multislides, and their inverses.\end{thm}
In view of the goal of establishing smooth 4-manifold invariants taking crown diagrams as input, it would be useful to pare down that list. The diagrammatic moves seem at first to be diverse and complicated, but it turns out there are relations between the moves in the sense that there are many sequences of moves for which it is possible to eliminate, say, all (de)stabilizations (Theorem \ref{stabthm}), and one may unify some of the above moves into a move of a single type (Proposition \ref{genshift}) or into a sequence of moves of a rather simple and familiar type called a \emph{slide}, with moves corresponding to restrictions on the choices available to build such a sequence (Theorem \ref{slidethm}). With these results in mind, establishing any smooth 4-manifold invariant calculated using crown diagrams of a fixed genus largely reduces to proving invariance under slides.

Descriptions of how the four moves affect a crown diagram were first given by Behrens and Hayano as applications of more general results in \cite{H,BH}. It will be useful for the reader to also see \cite[Section 2]{W2} for a unified, concise catalog of the moves on diagrams and their corresponding deformations, almost entirely following the previous two papers. The first result of this section involves a \emph{slide} between two elements $\gamma_i,\gamma_j\in\Gamma$, which resembles the handleslide move for Heegaard diagrams of 3-manifolds. To slide $\gamma_i$ over $\gamma_j$, choose an embedded path $\pi\co[0,1]\to\Sigma$ from a point in $\gamma_i$ to a point in $\gamma_j$, whose intersection with $\gamma_i$ and $\gamma_j$ is precisely its endpoints, and then replace $\gamma_i$ with the connect sum $\gamma_i\#\gamma_j$ in $\Sigma$ specified by that path.

\begin{thm}\label{slidethm}The handleslide and multislide moves affect $\Gamma$ by a sequence of slides. The shift move is also a sequence of slides, followed by a reordering of the elements of $\Gamma$.\end{thm}
\begin{proof}In \cite[Section~2.2]{H} there appears a \emph{surgery homomorphism} \[\Phi_c\co\mcg(\Sigma)(c)\to\mcg(\Sigma'),\] where $\mcg(\Sigma)(c)$ is the mapping class group of orientation-preserving diffeomorphisms of the genus $g$ oriented closed surface $\Sigma$ that fix the unoriented isotopy class of the embedded circle $c$, and $\mcg(\Sigma')$ is the mapping class group of the genus $g-1$ surface $\Sigma'$ (see also the introduction of \cite[Section~2.3]{B1}, where the map is called $\psi_\gamma$). According to \cite[Proposition 6.2]{BH} and \cite[Theorem 3.9]{H} respectively, any multislide or handleslide, applied using the pair $\gamma_1,\gamma_n\in\Gamma$, is realized by applying an element of $\Kr\Phi_{\gamma_1}\cap\Kr\Phi_{\gamma_n}$ to each element of $\{\gamma_2,\ldots,\gamma_{n-1}\}\subset\Gamma$, which we will call the \emph{nonstationary set}. In \cite{BH}, $\Kr\Phi_{\gamma_1}\cap\Kr\Phi_{\gamma_n}$ is denoted $\mathcal{K}(c,d)$, where $c=\gamma_1$ and $d=\gamma_n$.

For handleslides, $\gamma_1$ and $\gamma_n$ are disjoint and their union is nonseparating. Consider the genus $g-2$ surface $\Sigma''$ obtained from replacing $\gamma_1$ and $\gamma_n$ with pairs of disks whose centers are labeled $v_1,v_2$ for the disks coming from $\gamma_1$ and $w_1,w_2$ for the disks coming from $\gamma_n$. Hayano constructs an element of $\Kr\Phi_{\gamma_1}\cap\Kr\Phi_{\gamma_n}$ in \cite[Section~3]{H} as the lift (according to the pair of surgeries specified by the pair of 0-spheres $(v_1,v_2)$ and $(w_1,w_2)$) to $\mcg(\Sigma)$ of the point pushing map defined by an oriented embedded arc $\eta\subset\Sigma''$ connecting $v_i$ to $w_j$, for $i,j\in\{1,2\}$. With this lift understood, it is easy to see that the lemma is proved for handleslides. However, there is another way to see it: In \cite[Lemma~3.8]{H}, Hayano proves this lift can be factored into a product of mapping classes, each of which is a product of Dehn twists $t_{\tilde\delta(\eta)}\cdot t_{\gamma_1}^{-1}\cdot t_{\gamma_n}^{-1}$, where $\tilde\delta(\eta)$ is the obvious lift of the embedded circle $\partial\overline{\nu\eta}$ to $\Sigma$, where $\nu\eta$ is a regular neighborhood of $\eta$. One can then deduce the lemma by checking it for each of the two ways an embedded circle can intersect the triple $\tilde\delta(\eta),\gamma_1,\gamma_n$ as in Figure~\ref{hscheck}.
\begin{figure}
	\centering
	\begin{subfigure}{.4\linewidth}
	\centering
		\labellist
		\small\hair 2pt
		\pinlabel $\searrow$ at 45 50
		\pinlabel $\longrightarrow$ at 90 72
		\endlabellist
		\includegraphics[width=\linewidth]{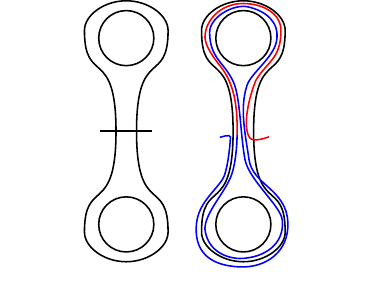}
		\caption{\ }
		\label{hscheck1}
	\end{subfigure}\hspace{1cm}
	\begin{subfigure}{.4\linewidth}
	\centering
		\labellist
		\small\hair 2pt
		\pinlabel $\longrightarrow$ at 52 72
		\pinlabel $\approx$ at 120 72
		\endlabellist
		\includegraphics[width=\linewidth]{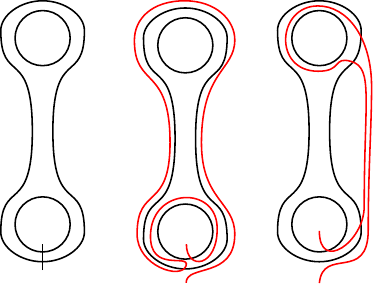}
		\caption{\ }
		\label{hscheck2}
		\end{subfigure}
\caption{Checking that the handleslide and multislide are sequences of slides. For handleslide, the two smaller circles are interchangeably $\gamma_1$ and $\gamma_n$ and the larger circle is $\tilde\delta(\eta)$. For multislide, $a$ is the lower circle, $\partial\Sigma^\circ$ is the upper circle, and the larger circle is $a\#_\zeta\partial\Sigma^\circ$.\label{hscheck}}
\end{figure}

For multislides, $\gamma_1$ and $\gamma_n$ intersect at one transverse point. By \cite[Lemma 5.4]{BH}, $\Kr\Phi_{\gamma_1}\cap\Kr\Phi_{\gamma_n}$ is generated by the mapping classes $\Delta_{\gamma_1,\gamma_n}=(t_{\gamma_1}t_{\gamma_n})^3$ and $\vartheta_{a,\zeta}=t_at_{a\#_\zeta\partial\Sigma^\circ}^{-1}$. The circles $a$ and $a\#_\zeta\partial\Sigma^\circ$ for the Dehn twists in the factorization of $\vartheta_{a,\zeta}$ are constructed as follows. Let $\Sigma^\circ\subset\Sigma$ be the complement of a regular neighborhood of $\gamma_1\cup\gamma_n$. Then $a$ is an arbitrary simple closed curve in the interior of $\Sigma^\circ$ and $a\#_\zeta\partial\Sigma^\circ$ is the simple closed curve given by the connected sum of $a$ and $\partial\Sigma^\circ$ along an embedded arc $\zeta$ whose interior is disjoint from $a$. It is straightforward to use Alexander's method \cite[Proposition 2.8]{FM} to show that $\Delta_{\gamma_1,\gamma_n}$ is isotopic to a half-twist about $\partial\Sigma^\circ$, and Figure \ref{delta} shows that such a mapping class has the effect of slides over $\gamma_1$ and $\gamma_n$ on any nonstationary circle. As for $\vartheta_{a,\zeta}$, if a nonstationary circle intersects $\zeta$, the effect of $\vartheta_{a,\zeta}$ is to apply a slide over $\partial\Sigma^\circ$ as specified by the subarc of $\zeta$ connecting that circle to $\partial\Sigma^\circ$ as in Figure \ref{hscheck1}, where $\partial\Sigma^\circ$ is the upper small circle and $a\#_\zeta\partial\Sigma^\circ$ is the larger circle. Such a slide is realized by two slides over $\gamma_1$ (and also by two slides over $\gamma_2$).
\begin{figure}
	\centering
    \begin{subfigure}{.4\linewidth}
	\centering
		\labellist
		\small\hair 2pt
		\pinlabel $\to$ at 83 34
		\pinlabel $\to$ at 173 34
		\endlabellist
        \includegraphics[width=\linewidth]{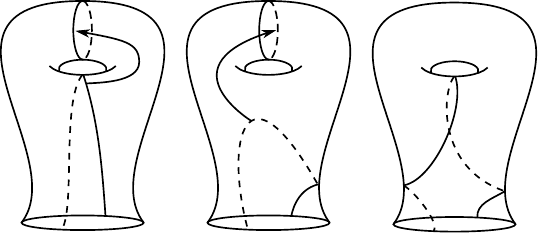}
        \caption{\ }
        \label{delta1}
    \end{subfigure}\hspace{1cm}
    \begin{subfigure}{.4\linewidth}
	\centering
        \includegraphics[width=\linewidth]{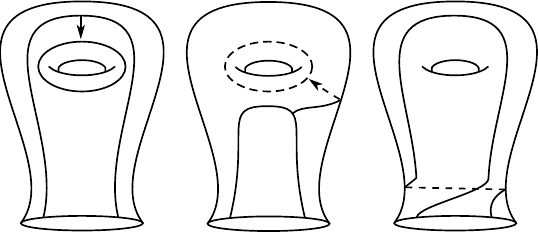}
        \caption{\ }
        \label{delta2}
    \end{subfigure}
    \caption{Checking that $\Delta_{c,d}$, a part of the multislide move, affects a nonstationary circle by a sequence of slides over $c$ and $d$.}\label{delta}
\end{figure}

The last case is the shift move. To make it easier to read, the rest of this proof uses the indexing conventions of \cite[Section 6.2]{BH}, in which $\Gamma=(\gamma_1,\ldots,\gamma_l)$, the initial fold merge occurs between the fold arcs corresponding to elements $\gamma_k,\gamma_l\in\Gamma$ for $1<k<l$ and the cusp merge involves the cusp between the fold arcs corresponding to $\gamma_1$ and $\gamma_l$. Like multislide, shift is a move one can perform for any pair $\gamma_1,\gamma_k$ that intersects at one transverse point, and according to \cite[Lemmas 6.3 and 6.6]{BH} it sends  
\begin{equation}\label{shifteq}(\gamma_1,\ldots,\gamma_k,\ldots,\gamma_l)\mapsto\left(\gamma_1,\ldots,\gamma_k,\gamma_l,\chi_i(\gamma_{k+1}),\ldots,\chi_i(\gamma_{l-1})\right),\end{equation} 
where $i$ can be chosen from $\{1,2\}$. Toward the end of the proof of \cite[Lemma 6.6]{BH}, the mapping class $\chi_i$ has the formula
\[\chi_i=\varphi_i\cdot(t_{\gamma_k}t_{\gamma_l}t_{\gamma_k})^{2m+1},\]
for some $m\in\Z$. Figure \ref{klk} shows that $(t_{\gamma_k}t_{\gamma_l}t_{\gamma_k})^{2m+1}$ affects nonstationary circles by a sequence of slides, so it remains to prove the same for $\varphi_i$. According to \cite[Lemma 6.4]{BH} and the paragraph directly above \cite[Lemma 6.4]{BH}, the map $\varphi_i$ is constructed as follows. First surger $\Sigma$ along $\gamma_l$, replacing that circle with two marked points $v_1,v_2$ to obtain a surface $\Sigma_{\gamma_l}'$. Since $\gamma_1$ and $\gamma_k$ each intersected $\gamma_l$ in $\Sigma$ unique transverse points, they appear as embedded arcs $\gamma_1',\gamma_k'$ in $\Sigma_{\gamma_l}'$, and each has the endpoints $v_1$, $v_2$. Together these arcs form a circle containing $v_1$ and $v_2$, and by slight perturbations they obtain two loops, $\eta_1\subset\Sigma_{\gamma_l}'\setminus\{v_2\}$ and $\eta_2\subset\Sigma_{\gamma_l}'\setminus\{v_1\}$, where $\eta_i$ is based at $v_i$ and is oriented so that it leaves $v_i$ along $\gamma_k'$ and returns along $\gamma_1'$. With that understood, \cite[Lemma 6.4]{BH} can be restated to say that, to apply $\varphi_i$, first replace $\gamma_l$ with a pair of marked disks to obtain $\Sigma_{\gamma_l}'$, then apply a point-pushing map along $\eta_i$ to the nonstationary circles (or arcs, if they intersected $\gamma_l$ in $\Sigma$), then perform surgery along $v_1,v_2$ to get $\Sigma$ back. This is clearly a collection of slides over $\gamma_l$.
\begin{figure}
    \begin{subfigure}{.45\linewidth}
		\centering
		\labellist
		\small\hair 2pt
		\pinlabel $\to$ at 83 150
		\pinlabel $\to$ at 170 150
		\pinlabel $\to$ at 255 150
		\pinlabel $\to$ at 170 30
		\pinlabel $\nearrow$ at 275 70
		\endlabellist 
        \includegraphics[width=\linewidth]{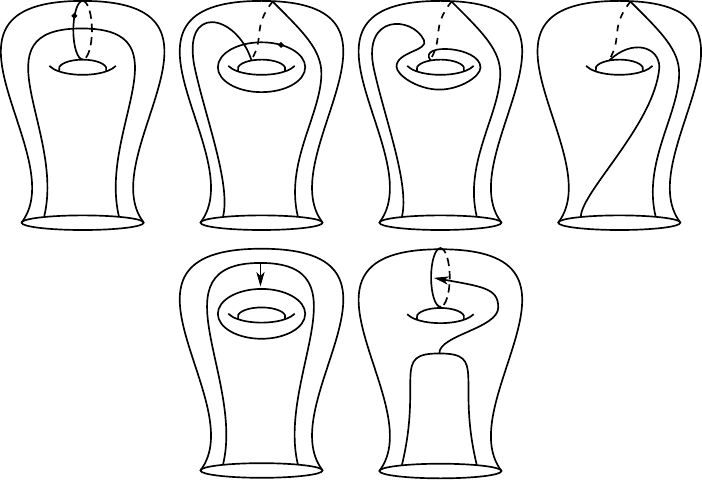}
        \caption{\ }
        \label{klk1}
    \end{subfigure}\hspace{.8cm}
    \begin{subfigure}{.45\linewidth}
	\centering
		\labellist
		\small\hair 2pt
		\pinlabel $\to$ at 83 150
		\pinlabel $\to$ at 170 150
		\pinlabel $\to$ at 255 150
		\pinlabel $\to$ at 170 30
		\pinlabel $\nearrow$ at 275 70
		\pinlabel $1$ at 28 218
		\pinlabel $2$ at 68 210
		\endlabellist
        \includegraphics[width=\linewidth]{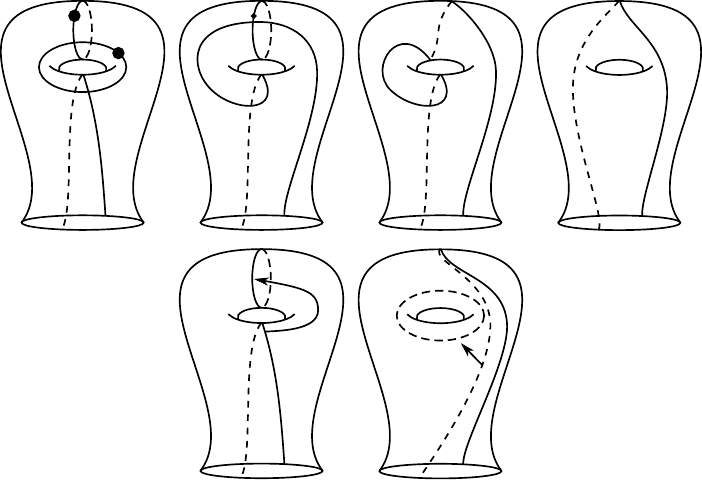}
        \caption{\ }
        \label{klk2}
    \end{subfigure}
    \caption{Realizing $(t_{\gamma_k}t_{\gamma_l}t_{\gamma_k})^{2m+1}$, part of the shift move, by slides over $\gamma_k$ and $\gamma_l$: (\ref{klk1}) is the case in which an arc of a nonstationary circle initially intersects $\gamma_k$; (\ref{klk2}) is the case in which it initially intersects $\gamma_l$.}
	\label{klk}
\end{figure}
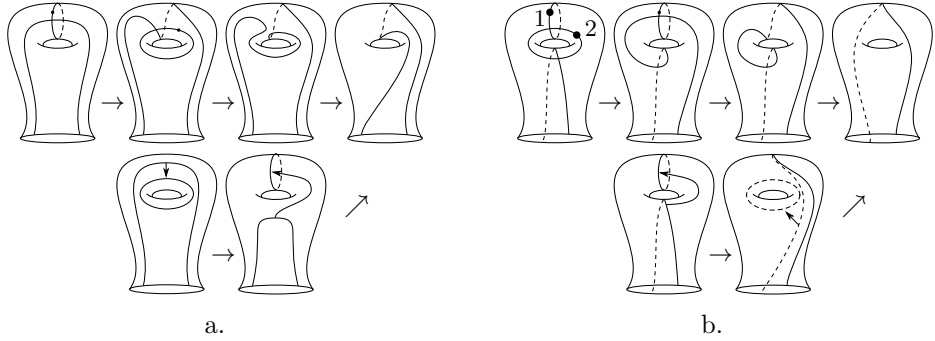
\end{proof}
\begin{rmk}\label{slidermk1}Performing an arbitrary slide on a crown diagram does not generally produce a new crown diagram. However, slides do correspond to actual handleslides in a handle decomposition corresponding to the crown diagram, in which fiber-framed 2-handles are attached to $\partial(\Sigma\times D^2)$ along the circles $\gamma_n\times\{e^{2n\pi i/k}\}$ (it is not relevant to this remark, but to recover $M$, this 2-handlebody $\mathcal{H}$ is capped off with a copy of $\Sigma'\times D^2$, where $\Sigma'$ has genus one less than $\Sigma$). In this interpretation of a slide, a point of the $i^{th}$ attaching circle for $1<i<n$ follows the path $\pi$ in $\Sigma\times\{e^{2n\pi i/k}\}$, then follows a path along $\{\pi(1)\}\times S^1$ toward $\gamma_1$ or $\gamma_n$, at which point the handleslide between 2-handles occurs. Notice there are two paths along $\{\pi(1)\}\times S^1$ to choose from, and the right choice is the one that lies in the $S^1$ interval corresponding to the nonstationary set. With this understood, it becomes clear that any collection of slides coming from Theorem \ref{slidethm} is of the type that allows the affected attaching circles (that is, those corresponding to the nonstationary set) to travel back to their respective fibers $\Sigma\times\{e^{2n\pi i/k}\}$ via ambient isotopy in $\mathcal{H}$. All of this is perhaps best visualized starting with a link diagram of $\Gamma$ in $\Sigma$, where the nonstationary set is $\{\gamma_2,\ldots,\gamma_{n-1}\}$ and $\gamma_i$ crosses under $\gamma_j$ for all $i<j$. Starting with a crown diagram enriched with this crossing data, $M$ can be recovered from the result of any slide as interpreted above.\end{rmk}
	
\begin{rmk}\label{slidermk2}Interpreting slides as traditional 4-dimensional handleslides, it is possible to strengthen Theorem \ref{slidethm} to say that the reordering of $\Gamma$ in the shift move as in Equation \ref{shifteq} is also given by a sequence of slides: To switch the order of a consecutive pair of circles, say $(\gamma_1,\gamma_2)$, it is a simple exercise to arbitrarily orient $\gamma_1$ and $\gamma_2$, then negatively slide $\gamma_1$ over $\gamma_2$ resulting in the pair of curves 
\[(\gamma_1-\gamma_2,\gamma_2),\] 
then positively slide the second curve over the first, resulting in 
\[(\gamma_1-\gamma_2,\gamma_2+\gamma_1-\gamma_2)=(\gamma_1-\gamma_2,\gamma_1),\] 
then negatively slide the first over the second to obtain 
\[(\gamma_1-\gamma_2-\gamma_1,\gamma_1)=(-\gamma_2,\gamma_1).\]\end{rmk}
\begin{figure}
	\centering{
		\labellist
		\small\hair 2pt
		\pinlabel 0 at 100 0
		\pinlabel 1 at 100 15
		\pinlabel 1 at 100 30
		\pinlabel 2 at 100 45
		\pinlabel 2 at 100 60
		\pinlabel 3 at 100 75
		\endlabellist
		\includegraphics{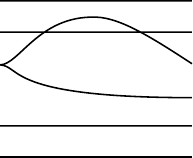}}
	\caption{A cerf graphic corresponding to a stabilization followed by handleslide in a Heegaard diagram.}\label{cerfdiag}
\end{figure}
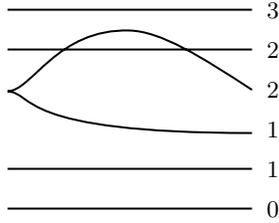
It is known that there are pairs of Heegaard splittings of closed 3-manifolds that require stabilization before they become isotopic. In other words, there are pairs of Heegaard diagrams $D,D'$ whose diffeomorphism types can be connected by a finite sequence of handleslides and isotopies of individual $\alpha$ and $\beta$ circles only after some positive number of stabilizations, even when $D$ and $D'$ have the same genus. One result below, Theorem \ref{stabthm}, states this is not the case for crown diagrams. Here are a few observations that might make this less surprising. In \cite{HTT} and every other construction known to the author, two Heegaard diagrams for the same 3-manifold coming from non-isotopic splittings of the same genus are related by switching the roles of the $\alpha$ and $\beta$ circles (equivalently, replacing a compatible Morse function $f\co Y\to[0,1]$ with $1-f$). This is a modification that does not exist for crown diagrams. It also seems that the obstruction to performing $R_2$ deformations discussed in \cite[Example 2.8]{W2} could be exploited to prove a similar non-isotopy result, working with the Cerf graphic of a carefully chosen deformation of Morse functions: Here a stabilization or its inverse occurs at every cusp, and handleslides occur precisely as instances of the mentioned obstruction (see \cite[Section 3]{La} for discussion of Cerf graphics). This is another kind of behavior that would not apply to crown diagrams. Finally, in the proof for Theorem \ref{stabthm} below, a crucial step is to move a stabilization to occur before or after any genus-preserving move by a homotopy of the deformation; such a technique is not generally available to Cerf graphics (see for example Figure \ref{cerfdiag}, a Cerf graphic corresponding to the stabilization of a genus-one Heegaard diagram, introducing circles $\alpha$ and $\beta$, immediately followed by a handleslide involving $\alpha$ or $\beta$. Such a modification is impossible there simply because of the number of attaching circles: No handleslides can be performed without first performing a stabilization).

One analogue of a Cerf graphic for a family of crown diagrams appears in Figure~\ref{cbd}. That figure depicts the critical set of a deformation, that is, a copy of $S^1\times[0,1]$ smoothly properly embedded into $M\times[0,1]$. Any deformation with this decorated critical surface is a stabilization deformation followed by handleslide and destabilization deformations (see also \cite[Figure 4 ]{W2}). The dots are swallowtails, the heavier arcs are cusp arcs, and the lighter arcs are the preimage of fold crossings in $S^2\times[0,1]$ (Figure \ref{hstrade1} is a depiction of the same critical surface omitting the cusps). See \cite[Sections 3.1 and 3.5.1]{W2} for an introduction to such figures and results concerning movements of their decorations. The rest of this paper freely uses ideas and terminology from those sections.
\begin{figure}[h]
\begin{center}\includegraphics{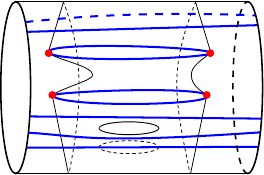}\end{center}\caption{The decorated critical surface $S^1\times[0,1]\subset M\times[0,1]$ for a stabilization deformation followed by handleslide and finally a destabilization at the same fold arc as the stabilization. Here, $t$ increases to the right.\label{cbd}}\end{figure}

We begin with a few results used in the proof of Theorem~\ref{stabthm}. The first implies that the location of a stabilization is irrelevant to the resulting crown diagram up to diffeomorphism.
\begin{prop}\label{stabmoveprop}Enumerate the fold arcs of an SPWF counter-clockwise according to $\Z_k$ as  $\varphi_1,\varphi_2,\ldots,\varphi_k$ and suppose $\alpha$ is a stabilization deformation whose initial flips occur at $\varphi_n$. Then $\alpha$ is homotopic, relative to its endpoints, to a stabilization whose initial flips occur on $\varphi_{n+1}$. \end{prop}
\begin{proof}In \cite[Figure~23]{W2} there appears a way to change the fold arc on which a flip occurs by a homotopy that fixes the endpoints of the given deformation. Use this to move the swallowtails into the desired region of fold points in $\crit\alpha$. The immersion arcs can be pushed in a genus-increasing manner so that the $R_2$ deformation that concludes the stabilization deformation occurs between the fold arc bounded by the two swallowtails and some other fold arc: This movement of the immersion locus merely corresponds to moving cusps into the higher-genus side of an intersecting fold arc at earlier values of $t$ than in the original deformation.\end{proof}

\begin{prop}\label{sdprop}If $\alpha$ is a stabilization deformation followed by a destabilization deformation, then $\alpha$ is homotopic to a constant deformation.\end{prop}
\begin{proof}Use Proposition~\ref{stabmoveprop} to arrange for the stabilization to occur at the same fold arc as the destabilization, and use \cite[Lemma 3.13]{W2} to isotope the two immersion arcs corresponding to the destabilization forward and backward in $t$ so that their $t$-minimal points (the pair of points at which the second Reidemeister-II fold crossing occurs) lie in the same fold arc as the $t$-maximal points coming from the destabilization. In this way, the Reidemeister-II fold crossings for the stabilization and destabilization occur between the same pair of fold arcs; Figure~\ref{sdpropfig} gives base diagrams for the resulting deformation.
\begin{figure}[b]
\begin{center}
\labellist
\small\hair 2pt
\pinlabel $\varphi$ at 96 70
\pinlabel $\overset{1}{\rightarrow}$ at 65 48
\pinlabel $\underset{6}{\leftarrow}$ at 65 38
\pinlabel $\overset{2}{\rightarrow}$ at 168 48
\pinlabel $\underset{5}{\leftarrow}$ at 168 38
\pinlabel $\overset{3}{\rightarrow}$ at 255 48
\pinlabel $\underset{4}{\leftarrow}$ at 255 38
\pinlabel $\bullet\ p$ at 320 5
\pinlabel $q\ \bullet$ at 191 65
\endlabellist
\includegraphics{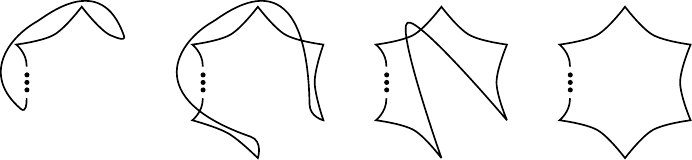}\end{center}\caption{Base diagrams for a stabilization deformation followed by a destabilization deformation. Steps 1 and 6 involve flips, the rest are 2-parameter crossings between the fold arc $\varphi$ and others.\label{sdpropfig}}\end{figure}
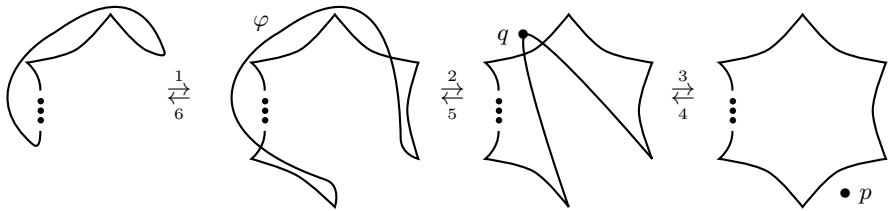

As arranged in the previous paragraph, near the end of the stabilization deformation (at step 3), a fold arc $\varphi$ cancels its two intersections with another fold arc in a genus-increasing $R_2$ deformation. In step 4, $\varphi$ re-acquires its intersections at the beginning of the destabilization deformation. During steps 3 and 4, a reference fiber chosen above $p$ with reference arcs crossing $\varphi$ gives a surgered crown diagram that is unchanged during those steps (see \cite[Remark 1.7]{W2} or \cite[Section 6.1]{BH} about surgered crown diagrams). For this reason, the interval in $t$ between steps 3 and 4 can be eliminated by homotopy of $\alpha$, resulting in a new deformation with steps 1-2-5-6. Going further, steps 2 and 5 can similarly be eliminated, using the same argument with one additional reference path from the same reference fiber that goes around the back of the sphere to the point $q$. What remains is a pair of flipping moves, followed by a pair of inverse flipping moves, which can also be eliminated by homotopy of $\alpha$. To explain this last step, choose a smooth path $\gamma$ in the fold locus of $\alpha$ connecting two of the swallowtail points such that $\gamma(t)\in\crit\alpha_t$. For example, $\alpha(\gamma)$ could coincide with a crossing introduced by one flip at the beginning of $\alpha(\gamma)$, (this crossing is then eliminated by the inverse flip at the end of $\alpha(\gamma)$). According to the local model of a swallowtail as a deformation of maps $B^4\to B^2$, the 2-cusped loop introduced at the beginning of $\gamma$ bounds a one-parameter family of disks in $S^2$, each of which is the image of a 4-ball neighborhood of a point of $\gamma$. These 4-balls trace out a neighborhood $U$ of $\gamma$ such that the restriction $\alpha|_U$ is simply a flip followed by its inverse. Such a deformation, called a \emph{detour} in this paper and in \cite[Definition 2.10]{W2}, is clearly a nullhomotopic loop in the space of maps $M\to S^2$.\end{proof}

\begin{prop}\label{dsprop}If $\alpha$ is a destabilization deformation followed by a stabilization deformation then $\alpha$ is homotopic to a sequence of handleslide deformations.\end{prop} 
\begin{proof}Apply Proposition~\ref{stabmoveprop} to arrange for the stabilization to occur at the same fold arc as the destabilization. Base diagrams for the result start at the right of Figure \ref{sdpropfig}, then proceed 4-5-6-1-2-3. Now apply the modification from \cite[Figure 21]{W2} to cancel the four swallowtails. The result is what is called a \emph{genus-decreasing immersion pair} in \cite[Definition 3.11]{W2}. Use \cite[Lemma 3.16(b)]{W2} to break this pair into many smaller genus-decreasing pairs, each of which is disjoint from the cusp locus. Because the supports of each pair of resulting $R_2$ deformations are disjoint (they occur in the preimages of distinct regions in $S^2$), these $R_2$ deformations can be performed in sequence, yielding a sequence of handleslides.\end{proof}

\begin{thm}\label{stabthm}Suppose two crown diagrams $(\Sigma,\Gamma)$ and $(\Sigma,\Gamma')$ are related by one of the following sequences of moves.\begin{enumerate}\item\hypertarget{1} Stabilize, then perform genus-preserving moves, then destabilize.\item\hypertarget{2} Destabilize, then perform genus-preserving moves, then stabilize.\end{enumerate}
Then there is a sequence of genus-preserving moves relating $(\Sigma,\Gamma)$ and $(\Sigma,\Gamma')$.\end{thm}
\begin{proof}For both cases the argument is to prove it is possible to switch the order in which the model deformations for (de)stabilizations and genus-preserving moves appear  in $\alpha$ without changing the fibration maps $\alpha_0$ and $\alpha_1$, then to apply Propositions~\ref{sdprop} and \ref{dsprop}.

Case \hyperlink{1}{1}. First, the argument for handleslide. As remarked in the proof of Theorem~\ref{slidethm} and \cite{H,BH,W2}, these moves (using vanishing cycles $\gamma_1$ and $\gamma_n$) affect a \emph{nonstationary set} of vanishing cycles $\{\gamma_2,\ldots,\gamma_{n-1}\}$ by a lift to $MCG(\Sigma)$ of a pushing map in $MCG(\Sigma')$, where $\Sigma'$ is the result of capping off $\Sigma\setminus(\gamma_1\cup\gamma_n)$ with two disks (one disk for multislide and shift). The pushing maps come from a choice of properly embedded arc (call it $\tilde{c}$) in the surface $\Sigma\setminus(\gamma_1\cup\gamma_n)$. The move affects an element of the nonstationary set by sliding over $\gamma_1$ or $\gamma_n$ (or both) along paths that are parallel to $\tilde{c}$, and a nonstationary vanishing cycle is modified by the move exactly when its image in $\Sigma\setminus(\gamma_1\cup\gamma_n)$ intersects $\gamma_1\cup\gamma_n\cup\tilde{c}$. For this reason, copies of the image of $\tilde{c}$ will appear in any nonstationary vanishing cycle that is modified by the move. Since the crown diagram destabilizes after the genus-preserving 
move, there is a fold arc $\varphi$ that sweeps across all of the higher-genus region of $S^2$ in a genus-decreasing $R_2$ deformation. Because of this, for some chosen reference fiber in the higher genus side of the fibration just before destabilizing, Theorem \ref{isotthm} implies the vanishing cycle $\gamma_i$ of this fold arc is disjoint from all the other vanishing cycles except $\gamma_{i\pm1}$ up to isotopy. Use Proposition~\ref{stabmoveprop} to cause $\varphi$ to also be the fold arc that sweeps across the back of $S^2$ in the stabilization deformation. Now $\gamma_i$ is disjoint from the image of $\tilde{c}$ and all other vanishing cycles except $\gamma_{i\pm1}$ before and after the handleslide. For this reason, the supports of the stabilization deformations and the handleslide deformation are disjoint and the deformations' effects on $\Gamma$ commute: the handleslide can be moved into the lower-genus side of either of the stabilization deformations. 

The multislide deformation similarly corresponds to a choice of loop $\tilde{c}$ based at a point in the circle $\partial(\Sigma\setminus(\gamma_1\cup\gamma_n))$, along which nonstationary vanishing cycles are slid over $\gamma_1$ and $\gamma_n$, in addition to the half-twists about $\partial(\Sigma\setminus(\gamma_1\cup\gamma_n))$ considered in Figure \ref{delta}. As before, use Proposition~\ref{stabmoveprop} to cause both the stabilization and destabilization to occur on a fold arc such that the image of its vanishing cycle in $\Sigma\setminus(\gamma_1\cup\gamma_n)$ is disjoint from the circles $\tilde{c}$ and $\partial(\Sigma\setminus(\gamma_1\cup\gamma_n))$, allowing the multislide deformation to move into the lower-genus side of, say, the stabilization deformation.

As for shift, essentially the same argument applies, except that there is no choice of curve in $\Sigma\setminus(\gamma_1\cup\gamma_n)$ -- the only requirement is for $\varphi$ to not be one of the fold arcs involved in either merge move. Once this is achieved, apply Proposition~\ref{sdprop} to cancel the pair of stabilization deformations. This concludes Case \hyperlink{1}{1}.

Case \hyperlink{2}{2}. This is easier than the previous case, because it is a genus-increasing modification of the deformation in the sense that the genus-preserving moves are all transferred into higher-genus diagrams. There is initially a vanishing cycle $\gamma_i$, disjoint from all others except $\gamma_{i\pm1}$, whose corresponding fold sweeps across $S^2$ in a genus-decreasing direction for the destabilization. After the destabilization deformation (following \cite[Theorem 4.1]{H}, in which $\gamma_i$ is called $d$), $\gamma_i$ survives as the endpoints of a short arc (which we will call $a$, but Hayano calls $\alpha$) which is transverse to the vanishing cycle on which the inverse flipping moves occurred: Choosing a horizontal distribution for the fibration structure over $[0,1]\times S^2$ given by $\alpha$, this fiberwise pair of points is precisely the set of points that runs into $\varphi$ in the direction $-\partial/\partial t$. Generically, the support of the genus-preserving moves is disjoint from $a$, so the deformation for the genus-preserving move may be moved backward in $t$ to occur before the destabilization deformation. Once this is achieved, apply Proposition~\ref{dsprop} to the remaining destabilization-stabilization deformation.\end{proof}

\begin{rmk}Theorem \ref{stabthm} along with Remarks \ref{slidermk1} and \ref{slidermk2} demonstrate a significant reduction in the complexity of the problem of defining general smooth 4-manifold invariants, from considering the incredible variety of Kirby diagrams to that of two well-studied moves -- slides and isotopies of $\Gamma$ in an enhanced crown diagram of a fixed genus and homotopy class.\end{rmk}

The next two results of the paper eliminate handleslides in favor of stabilizations and consolidate the other two genus-preserving moves.\begin{prop}\label{hstradeprop}Any handleslide deformation $\alpha$ is homotopic to a stabilization deformation followed by an inverse stabilization deformation, so that handleslides may be omitted from the list in Theorem~\ref{p2thm}.\end{prop}\begin{proof}Using language from \cite{W2}, the handleslide deformation is defined to be a genus-decreasing $R_2$ deformation that introduces two intersection points in the critical image, immediately followed by another $R_2$ deformation that cancels the intersections. The first step in the proposed homotopy of $\alpha$ is to smoothly insert a deformation into $\alpha$ at a $t$-value $t_0$ immediately before the initial $R_2$ deformation occurs: A stabilization deformation, immediately followed by its reverse. The inserted deformation, called a \emph{detour} in \cite[Definition 2.10]{W2}, forms a loop in the space of maps $M\to S^2$ (based at $\alpha_{t_0}$) which is clearly nullhomotopic, so the proposed modification can be achieved by a homotopy of $\alpha$. As in Case 2 of Theorem~\ref{stabthm}, move the handleslide deformation into the higher-genus region between the two stabilizations, resulting in Figure~\ref{hstrade1}.
\begin{figure}
	\centering
	\begin{subfigure}{.2\linewidth}
		\centering
		\includegraphics[width=\linewidth]{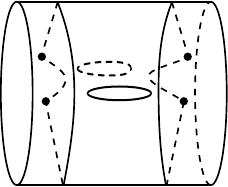}
		\caption{\ }
		\label{hstrade1}
	\end{subfigure}\hspace{0.8cm}
	\begin{subfigure}{.2\linewidth}
		\centering
		\includegraphics[width=\linewidth]{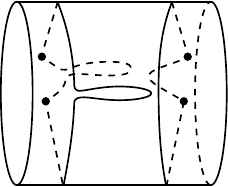}
		\caption{\ }
		\label{hstrade2}
	\end{subfigure}\hspace{0.8cm}	\begin{subfigure}{.2\linewidth}
		\centering
		\includegraphics[width=\linewidth]{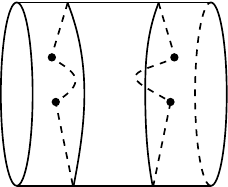}
		\caption{\ }
		\label{hstrade3}
	\end{subfigure}
	\caption{Including a handleslide deformation into a stabilize-destabilize pair.}\label{hstrade}
\end{figure}
The first stabilization deformation ends with an $R_2$ deformation canceling intersections between two fold arcs, and the handleslide deformation begins with an $R_2$ deformation that creates two intersections between a pair of fold arcs. It is possible to choose the first stabilization above so that these pairs of fold arcs coincide. With this understood, consider the image of these fold arcs in $S^2$ at the end of the stabilization and the beginning of the handleslide: It is the contraction and immediate re-formation of a lune of regular values bounded by those two fold arcs. The next step is a homotopy of $\alpha$ in which this lune disappears for decreasing $t$-increments until the two fold arcs do not separate at all (at least, until the $R_2$ deformation coming from the end of the handleslide deformation). This results in Figure~\ref{hstrade2}, which becomes clearly a stabilization-destabilization pair after possibly applying \cite[Proposition 3.10]{W2}, yielding Figure~\ref{hstrade3}.\end{proof}

In \cite[Section 4.1.2]{BH}, there appeared a \emph{generalized shift deformation}. Proposition~4.5 of the same paper gave its effect on a crown diagram. The following proposition can also be interpreted as a converse of \cite[Remark~2.9]{W2}.
\begin{prop}\label{genshift}The shift and multislide deformations in Theorem~\ref{p2thm} are homotopic to sequences of generalized shift deformations or length 1 and 2, respectively.\end{prop}
\begin{proof}Since shift deformations are a special case of generalized shift deformations, it remains to show that any multislide deformation can be decomposed into a pair of generalized shift deformations. This is easy to arrange: At a point $t_0$ just after the initial fold merge, which produces two cusps $c_1$ and $c_2$, perform an arbitrary cusp merge between $c_1$ and a cusp which is adjacent to $c_2$, then immediately reverse this cusp merge, and continue on with the deformation as it proceeded originally. The first and second pair of merge moves now each form a generalized shift deformation. The inserted deformation is a detour, so the proposed modification can be achieved by a homotopy of $\alpha$.\end{proof}

In \cite{H}, Hayano describes how a stabilization deformation affects a crown diagram. The description explicitly makes use of the fold arc on which the initial flips occur, though according to the following proposition, one can change this choice of fold arc without changing the diffeomorphism class of the resulting diagram.

The last result gives a simple criterion for when a crown diagram destabilizes.
\begin{prop}The crown diagram $(\Sigma,\Gamma)$ destabilizes if and only if $\Gamma$ has an element $\gamma_i$ which is disjoint from all other elements except $\gamma_{i\pm1}$.
\end{prop}
\begin{proof}In this case, \cite[Propositions 2.5(4) and 2.7]{W2} imply the existence of a deformation corresponding to steps 4 and 5 of Figure~\ref{sdpropfig}. The vanishing cycles measured from a reference fiber on the higher-genus sides of a pair intersecting folds are disjoint, and the vanishing cycles on either side of a cusp must intersect transversely at a single point so the vanishing cycles for the three fold arcs in each of the two triangles after step 5 (or, indeed, any such triangle) appear as in \cite[Figure 5]{W1} and so are suitable for inverse flipping moves that complete the destabilization.
\end{proof}

\end{document}